\documentclass[11pt, reqno]{amsart}
\usepackage[utf8]{inputenc}
\numberwithin{equation}{section}
\usepackage[usenames, dvipsnames]{color}
\usepackage[shortlabels]{enumitem}
\usepackage[hidelinks]{hyperref}
\usepackage{cleveref}
\usepackage{comment}
\usepackage{todonotes}
\newcommand{\qtq}[1]{\quad\text{#1}\quad}
\usepackage{mathabx}
\usepackage{amsmath, amssymb, amsfonts, amsthm, mathtools}
\usepackage{dsfont}
\usepackage[scr=boondoxo]{mathalfa}
\usepackage{exscale}
\usepackage{cite}
\usepackage{epsfig}
\usepackage{amscd}
\usepackage{graphics}
\usepackage{listings}

\usepackage{lineno}

\newcommand{\R}{\mathbb{R}}

\newcommand{\T}{\mathbb{T}}
\newcommand{\C}{\mathbb{C}}
\newcommand{\N}{\mathbb{N}}
\newcommand{\Z}{\mathbb{Z}}

\usepackage{amsmath,amssymb,amsthm}

\newtheorem{thm}{Theorem} 
\newtheorem{cor}[thm]{Corollary}
\newtheorem{conj}[thm]{Conjecture}
\newtheorem{pro}[thm]{Proposition}
\newtheorem*{claim*}{Claim}

\newtheorem{lem}[thm]{Lemma}
\newtheorem{remark}[thm]{Remark}

\theoremstyle{remark} 

\theoremstyle{definition} \newtheorem{dfn}[thm]{Definition}

\numberwithin{equation}{section} \numberwithin{thm}{section}

\newcommand{\eps}{\varepsilon}

\DeclareMathOperator{\F}{\mathcal{F}}
\DeclareMathOperator{\K}{\mathcal{K}}

\newcommand{\vphi}{\varphi}

\newcommand{\bone}{\mathds{1}}

\newcommand{\jbrak}[1]{\langle#1\rangle}

\renewcommand{\d}{\mathrm{d}}

\newcommand{\bbar}{\overline}

\renewcommand{\tilde}{\widetilde}
\renewcommand{\hat}{\widehat}

\renewcommand{\d}{\mathrm{d}}

\newcommand{\dx}{\, \mathrm{d}x}

\newcommand{\dt}{\, \mathrm{d}t}
\newcommand{\dr}{\, \d r}

\newcommand{\dtau}{\, \mathrm{d}\tau}

\newcommand{\cQ}{\mathcal{Q}}

\newcommand{\cD}{\mathscr{D}}

\newcommand{\defe}{\overset{\mathrm{def}}{=}}

\newcommand{\supp}[1]{\mathrm{supp}(#1)}
\renewcommand{\rm}[1]{\mathrm{#1}}

\newcommand{\bk}{\mathbf{k}}

\newcommand{\lo}{\mathrm{lo}}

\newcommand{\hi}{\mathrm{hi}}

\begin{document}

\title[Periodic stationary waves for dispersion-managed NLS]{On analyticity of solitons for the periodic dispersion-managed Nonlinear Schr\"odinger Equation}

\author[S. Herr]{Sebastian Herr} \address[S. Herr]{Fakult\"at f\"ur
  Mathematik, Universit\"at Bielefeld, Postfach 10 01 31, 33501
  Bielefeld, Germany}
\email{herr@math.uni-bielefeld.de}

\author[D. Hundertmark]{Dirk Hundertmark} \address[D. Hundertmark]{Fakult\"at f\"ur Mathematik, Karlsruher Institut f\"ur Technologie, 76128 Karlsruhe, Germany}
\email{dirk.hundertmark@kit.edu}

\author[J.L. Marzuola]{Jeremy L. Marzuola} \address[J.L. Marzuola]{Department of Mathematics, University of North Carolina, Chapel Hill, NC, 27599}
\email{marzuola@math.unc.edu}

\author[T. Van Hoose]{Tim Van Hoose} \address[T. Van Hoose]{Department of Mathematics, University of North Carolina, Chapel Hill, NC, 27599}
\email{tvh@unc.edu}

\begin{abstract}

We explore the regularity of energy maximizers for the Lagrangian of a periodic dispersion managed fiber optic at fixed intensity, on a torus of length $L$ and with vanishing average dispersion.  We show that the Fourier coefficients decay at a polynomial rate, and then upgrade this to exponential decay, so that the maximizers are analytic in space for large enough $L$.  In addition, by an asymptotic comparison to the optimizers on the real line, we prove that the solutions are non-trivial.  We also consider a conjecture that the maximizer necessarily has an underlying symmetry inherent to both the energy functional and the resulting Euler-Lagrange equation.  All of our results are supported with illustrative numerical experiments.  
\end{abstract}

\subjclass[2000]{} \keywords{}

\maketitle

\section{Introduction}\label{sect:intro}
\noindent
Consider the following one-dimensional nonlinear Schr\"odinger equation: 
    \begin{equation}\label{eq:disp-man-nls}
        iv_t+d(t)v_{xx}+C|v|^2v=0 \qtq{on} \R_t \times \T_{L}, \ \ v(x,0) = v_0 \in L^2 (\T_{L})
    \end{equation}
where $\T_L = \R /(L\Z)$ is the torus of length $L$. Solutions to this equation naturally arise in the study of finite fiber optics, where the dispersion of the cable is periodically varied along its length, which is encoded in the function $d(t)$.  Indeed, in the derivation of this formulation from a paraxial approximation, the time variable $t$ in \eqref{eq:disp-man-nls} is related to spatial location in the fiber, while our spatial variable $x$ corresponds to the time component.  We are actually interested in more than just a generic dispersion $d(t)$; our main interest will be in the regime of \emph{strong dispersion management}, for which we will take the dispersion
\begin{equation}
    d(t) = d_{\rm{av}} + \frac{1}{\eps}d_0\left(\frac{t}\eps\right),
\end{equation}
where $d_{\rm{av}}$ is the average of $d(t)$ over one period (which we generically take to be length $1$), and the function $d_0$ has mean zero over its period. Inserting this choice into the equation \eqref{eq:disp-man-nls} and changing variables, we arrive at the envelope equation 
\begin{equation}\label{eq:disp-man-nls-env}
iv_t+d_0(t)v_{xx}+\eps(d_{av}v_{xx}+|v|^2v)=0.
\end{equation} 
In what follows, we will assume that $d_0(t)$ is piecewise constant and $1$-periodic, and equal to $1$ on $[0, \tfrac12]$ and equal to $-1$ on $[\tfrac12, 1]$.

With our specific choice of $d_0$, an averaging process (detailed in \cite[Section 2]{CHL23}, \cite{gabitov1996averaged, gabitov1996breathing}) yields the Gabitov–Turitsyn equation
\begin{equation}\label{eq:av-nls}
iv_t+\eps d_{\rm{av}}v_{xx}+\eps Q(v)=0,
\end{equation}
where $T_r=e^{ir\partial_x^2}$, $q(v) = |v|^2 v$ and
\begin{equation}\label{eq:Q}
Q(v)(t,x)=\int_0^{1}T_r^{-1}(q(T_rv))(t,x) dr,
\end{equation}
which comes from averaging \eqref{eq:disp-man-nls-env}.
 In the periodic case, such a derivation was recently carried out in the dissertation \cite{adekoya2019periodic}.

In this work, we study the existence and regularity of periodic dispersion managed solitary waves, which are spatially periodic solutions (with period $L > 0$) of the form $e^{it\omega} u(x)$.

At this point, we will state our main theorems somewhat informally, and postpone a rigorous statement for the sequel. Our results are as follows:
\begin{thm}[Informal]
    Set $d_{\rm{av}} = 0$, $\omega>0$, and let $L >0$ be sufficiently large. The spatial Fourier modes of stationary solutions to \eqref{eq:av-nls}
    decay exponentially. In particular, such solutions are analytic in space.  Further, for large enough $L$ the maximizers are nontrivial. 
\end{thm}

We will now give a more in-depth sketch of the context of our main results in the literature, as well as the method of proof.  Recall that in the paraxial limit, the variable that plays the role of $x$ in \eqref{eq:disp-man-nls}-\eqref{eq:av-nls} is actually time.  Hence, we are considering time-periodically generated signals across an array of periods.

In \cite{EHL11}, it is shown that stationary solutions to \eqref{eq:disp-man-nls} with $v \in L^2 (\R)$ decay exponentially in both space and frequency; the non-triviality of such objects is to some extent trivial since constants are not in $L^2 (\R)$.  For the case where the nonlinearity is quintic ($q(v) = |v|^4 v$) and the integration domain is infinite ($r \in [0,\infty)$), stationary solutions are the optimizers of the Strichartz inequality and are known to be Gaussians in $1$ dimension, see \cite{foschi2007maximizers,hundertmark2006sharp}.  In the cubic case and a finite time interval $r \in [0,1]$, the stationary states are known to have oscillatory tails, see \cite{lushnikov2004oscillating}.

To be more explicit, for $u \in L^2(\R)$, the objects of study are stationary solutions to the equation 
\begin{equation}\label{E:0av-weak-dmnls}
    \omega u = Q(u),
\end{equation}
where $Q(u) = Q(u,u,u)$ is a nonlocal trilinear operator defined by 
\begin{equation*}
    Q(f,g,h) = \int_0^1 T_r^{-1} (T_r f \bbar{T_r g} T_r h) \dr,
\end{equation*}
where $T_r = e^{ir\partial_x^2}$ is the free Schr\"odinger propagator. We notice that the Lagrange multiplier $\omega$ naturally scales with the mass $\|u\|_{L^2}^2$ of the solution. That is to say, if we consider solutions to \eqref{E:0av-weak-dmnls} with mass $\lambda$, the corresponding value of $\omega$ is given by $\omega = \omega_1 \lambda$, where $\omega_1$ is given as the extremizer of an optimization problem over functions of mass $1$. 

For the case $v(t,\cdot) \in L^2 (\R)$ for all $t$, several recent works have analyzed solutions of \eqref{eq:disp-man-nls}, \eqref{eq:av-nls}.  Indeed, the time dependent equation \eqref{eq:av-nls} has seen significant recent progress in this setting.  In the case when $d_{\rm av} \neq 0$, see for instance \cite{murphy2022modified,murphy2024note,choi2022averaging,choi2024continuum,choi2025scattering,kawakami2024small,kowalski2025ill}.  For both time independent and dependent results in higher spatial dimensions, see \cite{choi2025ground,campos2024averaging}.

The closest related works to ours consider analysis of solutions to \eqref{E:0av-weak-dmnls} with $u \in L^2 (\R)$, i.e. solitary wave solutions with zero average dispersion ($d_{\rm av} = 0$).  This problem has been studied significantly as well, see \cite{CHL23,EHL11,ZGJT-cc,hundertmark2015stability,HLRZ17,HLRZ18,HL-var,HL12b,HL09}. See also the work of \cite{green2016exponential}, which extends the exponential decay results to $d_{\rm av} \geq 0$.  The methods developed in \cite{HL09} and \cite{EHL11} explore a form of ellipticity present in the multilinear interactions given by $Q(u)$ that we will exploit here as well.  

On the torus, we perform a similar analysis; we solve the same problem:
\begin{equation}\label{E:0av-dmnls-periodic}
    \omega f = Q(f,f,f) \qtq{on} \T_L,
\end{equation}
and again note that the Lagrange multiplier $\omega$ still scales with the mass of $f$, but also depends implicitly on the length of the torus $L$. As a result, we will normalize our solutions to always have fixed $L^2(\T_L)$-norm equal to $1$, and work with $\omega = \omega_1(L)$, which is uniformly bounded in $L$.  Our main contribution is to utilize the spatial scale $L$ of the torus to get precise bilinear estimates in the periodic setting to give $L$ dependent polynomial decay bounds in frequency.  

Our method will be to look for maximizers to the $L^4([0,1] \times \T_L)$ Strichartz norm of the form 
\begin{equation}\label{maxdef}
   \omega_1 (L) = \max_{\|f\|_{L^2(\T_L)} = 1} \cQ(f,f,f,f)
\end{equation}
where the functional $\cQ(f,f,f,f)$ is defined by 
\begin{equation}\label{E:qdef}
    \cQ(f_1, f_2, f_3, f_4) = \int_0^1 \int_{\T_L} \bbar{T_r f_1} T_r f_2 \bbar{T_r f_3} T_r f_4 \dx \dr.
\end{equation}

As we will need it later, we let $\omega_1^*$ be the corresponding constant on the real line, defined as
\begin{equation}\label{maxdefR}
   \omega_1^* = \max_{\|f\|_{L^2(\R)} = 1} \cQ(f,f,f,f),
\end{equation}
where here and in Section \ref{sec:largeL} $\cQ$ denotes the corresponding expression on $\R$ instead of $\T_L$,
which is finite, obviously positive, and the maximum is attained, see \cite{HL09} for instance.

In the periodic case, using the convention that
\begin{equation}
    \label{eqn:FTdef}
    f(x) =\sum_{k \in \mathbb{Z}} \hat{f} (k) e_L(kx)
\end{equation}
for {the orthonormal $e_L(x) = e^{\frac{2 \pi i x}{L}}/\sqrt{L}$},
we note that in terms of the Fourier coefficients $\hat{f_j}$ we have 
\begin{equation}\label{E:qdefFourier}
    \cQ(f_1, f_2, f_3, f_4) 
      = L^{-1}\sum_{\substack{\bk\in\Z^4\\k_1-k_2+k_3-k_4=0}}
        \!\!\!\!\!\!\!\int_0^1 
          e^{\frac{4\pi^2 i\Phi(\bk)r}{L^2}} 
        \dr \!\
         \bbar{\hat{f_1}(k_1)} \hat{f_2}(k_2) 
        \bbar{\hat{f_3}(k_3)} \hat{f_4}(k_4)
\end{equation}
with the phase function $\Phi(\bk)= k_1^2-k_2^2+k_3^2-k_4^2$.

Existence of spatially periodic solutions to \eqref{eq:disp-man-nls} was recently considered in \cite{adekoya2022maximisers}.  There, in Theorem $6.2$, the authors establish that given $T=1$, there exists a $C \approx 2.6$ such that for all $L \geq \frac{2 \pi}{C} \approx 3.9$ maximizers exist for any mass.

Note that $\omega_1(L)=O(1)$. Indeed, by definition  $\cQ(f,f,f,f)$ is  the fourth power of the  $L^4_{r,x}([0,1]\times\T_L)$-norm of $T_r f$, and since we can write 
\begin{equation*}
    \int_0^1 \int_{\T_L} |T_r f|^4 \dx \dr = \int_0^1 \int_{\T_L} |T_r f \bbar{T_r f}|^2 \dx \dr,
\end{equation*}
the bound 
$\cQ(f,f,f,f)\lesssim \|f\|_{L^2}^4$ follows from using the third case of the bilinear estimate \Cref{T:MV-Ltorus} below.
It will be convenient to rewrite the stationary equation on the Fourier side.
Testing \eqref{E:0av-dmnls-periodic} against a function $g=e_L(k\cdot)$ 
and using that
 \begin{equation}
      \langle e_L (k x), f \rangle = \hat{f}(k) \ \text{for} \ e_L(x) = L^{-\frac{1}{2}} e^{  \frac{2 \pi ix}{L} }
 \end{equation}
 as we will explain below in Section \ref{S:notation},
 we find that
 
\begin{equation}
\label{eq:fourier-eq}
    \omega(L)\hat{f}(k) = \frac{1}{L}\sum_{\substack{k_2,k_3,k_4 \\ k-k_2+k_3-k_4=0}}  \!\!\!\!\int_0^1 
                      e^{\frac{4\pi^2 i\Phi(k, k_2, k_3, k_4)r}{L^2}} 
                    \dr  
            \hat{f}(k_2) \bbar{\hat{f}(k_3)} \hat{f}(k_4).
\end{equation}

    With these definitions and context in hand, we can now state our theorems more precisely. First, we show
    that solutions have polynomial decay, at an order that grows as $L$ grows sufficiently large:
    \begin{thm}\label{thm:poly-decay}
        For any $L$ sufficiently large, the Fourier modes $\hat{f}(k)$ of the mass-$1$ maximizer of \eqref{maxdef} obey the decay estimate
        \begin{equation*}
            |\hat{f}(k)| \lesssim_L |k|^{-\log_3(\frac{\omega}{2K} L^{\frac{1}{2}})} \text{ for all } |k|>s_0,
        \end{equation*}
        for some $s_0=s_0(f,L)$, where $K > 0$ is a constant independent of $L$ and $\omega=\omega_1(L)$.
    \end{thm}
This is proved in Section \ref{sec:superpoly}.
In addition, we are able to prove exponential decay of the Fourier coefficients.
    \begin{thm}
    \label{thm:1.4}
        Let $L$ be sufficiently large. Then, the Fourier modes of every mass-$1$ weak solution $f$ of \eqref{E:0av-dmnls-periodic}, with $\omega=\omega_1(L)$
        decay exponentially. That is, there exists a constant $c > 0$ so that 
        \begin{equation*}
            |\hat{f}(k)| \lesssim e^{-c|k|}
        \end{equation*}
        for all $k\in \Z$. In particular, all such solutions are analytic. 
    \end{thm}
    This is proved in Section \ref{sec:exp}. There, the condition on the size of $L$ is made more explicit in terms of $\omega$.

    In both Theorems \ref{thm:poly-decay} and \ref{thm:1.4}, note that our estimates do not give a uniform decay rate in Fourier space.  For instance, one reason for this is translation invariance in the Fourier domain. As our main goal is to prove analyticity, this is sufficient, while an interesting consideration for future work is the width of that strip. 

    The requirement that $L$ be large is meaningful because we also prove in Section \ref{sec:largeL} that
    $$
    \liminf_{L\to \infty} \omega_1(L)\geq \omega_1^*>0,$$
see Proposition \ref{pro:unif-lower}.
    This implies the following:
\begin{thm}
\label{thm:nontrivial}
For $L$ sufficiently large, the maximizer for \eqref{maxdef}
is not the trivial constant solution $z_L=L^{-\frac12}$.
\end{thm}

Since the Schr\"odinger equation preserves symmetry (and anti-symmetry), these equations can be framed in the context of $L^2_s$ and $L^2_a$, the Hilbert spaces of symmetric and anti-symmetric functions respectively and related results can be shown to hold.  We conjecture that the global maximizer for functions constrained in $L^2(\mathbb{R})$ should be symmetric, but in general this shows that even modding out by the large number of symmetries present in this model, there are multiple non-trivial solutions to \eqref{E:0av-dmnls-periodic}.  Note, this is known to not be true for some values of $L$ as seen in \cite{adekoya2022maximisers}, Corollary $6.3$.  As an application of some of our numerical simulations, we will give supporting evidence for this conjecture by computing solutions to \eqref{E:0av-dmnls-periodic} on very large domains.  This result should be strongly related to a uniqueness result proven recently in \cite{choi2023uniqueness}.  

\begin{conj}
\label{conj:sym}
    The optimizer of \eqref{maxdefR} is even up to symmetries.  
\end{conj}

Similarly, we wish to explore the question of nontrivial solutions to \eqref{E:0av-dmnls-periodic} in the case of a torus of length $L$.  Note, we could also consider intervals with homogeneous Dirichlet or Neumann boundary conditions, but by doubling the domain with an even or odd reflection, we return to the periodic case.  In particular, we note that with periodic boundary conditions, we can always simply take a constant solution to \eqref{E:0av-dmnls-periodic}, but we wish to understand if we can find others through the variational interpretation of \eqref{E:0av-dmnls-periodic}.

The paper will proceed as follows.  We review all the necessary notation in Section \ref{S:notation}. Then, in Section \ref{sec:superpoly}, we review the bilinear estimate of Moyua and Vega in the form we need it, along with some refinements, and establish that as $L$ gets large, the Fourier modes of any weak solution to \eqref{E:0av-dmnls-periodic} 
must decay at a polynomial rate that is set by $\omega$ and $L$ and that improves with $L\to \infty$.  In Section \ref{sec:exp}, we modify arguments originated by the second author and collaborator to establish that the Fourier coefficients actually decay exponentially.  In Section \ref{sec:largeL}, we prove the lower bound for $\omega_1(L)$, which also implies Theorem \ref{thm:nontrivial}, and in addition we record an upper bound following Bourgain and Zygmund.  Finally, in Section \ref{sec:num}, we numerically explore solutions to \eqref{E:0av-dmnls-periodic}, demonstrating that indeed there are non-trivial solutions for $L$ sufficiently large and that in fact by restricting to various symmetry classes we can find at least two non-trivial solutions.

\section*{Acknowledgments}
The authors thank Jason Murphy, Vadim Zharnitsky and Almut Burchard for various helpful conversations about dispersion managed systems. JLM credits long-term visits to Karlsruhe Institute of Technology and to Universit\"at Bielefeld  as getting a number of the ideas presented here started and thanks KIT for hosting said visit.

SH acknowledges support from the Deutsche Forschungsgemeinschaft (DFG, German Research Foundation) -- Project-ID 317210226 -- SFB 1283.

DH acknowledges support from the Deutsche Forschungsgemeinschaft (DFG, German Research Foundation) -- Project-ID 258734477 – SFB 1173. 

JLM acknowledges support from the NSF through NSF CAREER Grant
DMS-1352353, NSF grant DMS-1909035,  NSF Applied Math Grant DMS-2307384. 

TVH was partially supported by NSF RTG grant DMS-2135998, NSF Applied Math Grant DMS-2307384, and a Simons Dissertation Fellowship.

        \section{Notation}
           \label{S:notation}
        In this section, we introduce some notation that will be used throughout the remainder of the paper.\par
        We write $A \lesssim B$ or $B \gtrsim A$ to denote the inequality $A \leq CB$ for some constant $C > 0$, where $C$ may depend on parameters like the dimension or the indices of function spaces. If $A \lesssim B $ and $B \lesssim A$ both hold, then we write $A \sim B$. Vectors in $\Z^d$ will be written in boldface, as in $\bk \in \Z^d$. For any $\alpha > 0$, we define the set $[\alpha] \subset \Z$ to be the set of integers 
        \begin{equation}
            [\alpha] \defe \{ -\lceil \alpha \rceil, \ldots, 0,\ldots, \lceil \alpha \rceil \}.
        \end{equation}  

        We define $\T_L$ to be the domain $\R / (L \Z)$ and generically take
        \begin{equation}
        \label{lpnorm}
        \| f \|_{L^p (\T_L)} = \left(  \int_0^L |f(x)|^p dx \right)^{1/p}.
        \end{equation}
        For notational brevity, we define $e_L(x) = L^{-\frac{1}{2}}e^{\frac{2\pi i x}{L}}$. In particular, the functions $\{e_L(kx)\}_{k \in \Z}$ form an orthonormal basis in $L^2(\T_L)$.
        For functions in $L^2 (\T_L)$, we can define the Fourier transform with the scaling
        \begin{equation}
        \label{fourier}
            \hat{f}(k) =\int_{\T_{L}} f(x) e_{L}(-kx)\dx, \ k \in \Z.
        \end{equation}
        We have the $L^2 (\T_L)$ inner product
                \[
                    \jbrak{f, g} := \int_{\T_L} \bbar{f(x)}g(x) \dx = \sum_{k \in\Z} \bbar{\hat{f}(k)} \hat{g}(k).
                \]
        and, with this convention, we have
        \begin{equation}
        \label{ftrep}
        f(x)  =  \sum_{k \in \Z} e_L(kx) \hat f (k).
        \end{equation}
        and the Plancherel identity reads as
        \begin{equation}
        \label{plancherel}
        \sum_{k \in \Z} | \hat f (k) |^2 = \int_0^L|f|^2 \dx,
        \end{equation}
        i.e. that the Fourier transform is a linear isometry $L^2(\T_L) \to \ell^2(\Z)$.    

        Using the normalizations for the Fourier transform established above, we have
    \begin{equation*}
        T_r f_j(x) = \sum_{k_j \in\Z}e_L(k_jx)e^{\frac{-4\pi^2 i k_j^2 r}{L^2}}\hat{f_j}(k_j).
    \end{equation*}
    Inserting this into the definition of $\cQ$, we have 
    \begin{equation*}
    \aligned
        & \cQ(f_1, \ldots, f_4) = \\
            & L^{-1}\sum_{k_1, \ldots, k_4 \in \Z} \int_0^1 \int_{\T_L}e^{\frac{4\pi^2 i\Phi(\bk) r}{L^2}} L^{-1}e^{\frac{2\pi i(-k_1 + k_2 - k_3 + k_4)x}{L}}\prod_{j=1}^4 \mathscr{C}^j \hat{f_j}(k_j) \dx\dr,
        \endaligned
    \end{equation*}
    where $\mathscr{C}$ denotes complex conjugation (so that $\mathscr{C}^j=\mathscr{C}$ for $j$ odd and  $\mathscr{C}^j=\mathrm{id}$ for $j$ even). 
        Performing the $x$ integral we see that
    \begin{equation*}
        \cQ(f_1, \ldots, f_4) =L^{-1} \sum_{k_1 - k_2 + k_3 -k_4 = 0} \int_0^1 e^{\frac{4\pi^2 i \Phi(\bk) r}{L^2}}\prod_{j=1}^4 \mathscr{C}^j \hat{f_j}(k_j) \dr.
    \end{equation*}
    The same algebraic manipulation as before reveals that when $k_1 - k_2 + k_3 - k_4 = 0$ that the phase $\Phi(\bk) := k_1^2 - k_2^2 + k_3^2 - k_4^2$ reduces to 
    \begin{equation*}
        \Phi(\bk)|_{k_1 - k_2 + k_3 - k_4 =0}= -2(k_1 -k_4)(k_3 - k_4),
    \end{equation*}

so that the phase degenerates exactly on the resonant sets $\{k_1=k_4\}$ and $\{k_3=k_4\}$. On their complement, we will be able to exploit the oscillations in the $r$-integral. The contribution of the resonant sets themselves is bounded by
$$ 3L^{-1}\prod_{j=1}^4\|f_j\|_{L^2},
$$
which is harmless.
    \section{The Bilinear Estimate and Polynomial Decay}
        \label{sec:superpoly}
            We first record the proof of \cite[Theorem 2]{moyuaBoundsMaximalFunction2008} adapted to the torus of sidelength $L$, which applies almost verbatim. 
            
            \begin{thm}\label{T:MV-Ltorus}
               Assume that $f,g\in L^2(\T_L)$ such that
                $$\rm{dist}(\supp{\hat{f}}, \supp{\hat{g}}) = M_0\geq 0.$$ 
                For $I = [0, T]$ with $T \leq 1$, we have
                \begin{equation*}
                    \left\| e^{it \partial_x^2} f \bbar{e^{it\partial_x^2}g}\right\|_{L_{t,x}^2 (I \times \T_L)} \lesssim (\Delta(M_0, T))^{\frac{1}{2}} \|f\|_{L^2}\|g\|_{L^2}
                \end{equation*}
                where 
                \begin{equation*}
                    \Delta(M_0, T) = 
                    \begin{cases}
                        \begin{aligned}
                            &\frac{T}{L} , \quad \text{if } M_0 \geq \frac{ L^2}{T},\\
                            &\frac{L}{ M_0}, \quad \text{if } L T^{-\frac{1}{2}} < M_0 < L^2 T^{-1}, \\
                            &T^{\frac{1}{2}}  \quad \text{if } M_0 \leq L T^{-\frac{1}{2}}.    
                        \end{aligned}
                        
                    \end{cases}
                \end{equation*}
            \end{thm}

            \begin{proof}
                We can write 
                \begin{align*}
                    e^{it\partial_x^2}f(x) &= \sum_{n \in \Z} a_n e^{-\frac{4\pi^2 i n^2t}{L^2}}e_L(nx), \; a_n=\hat{f}(n) \\
                    e^{it\partial_x^2}g(x) &= \sum_{k \in \Z} b_k e^{-\frac{4\pi^2 i k^2t}{L^2}}e_L(kx), \; b_k=\hat{g}(k)
                \end{align*}
                so that the product in the $L^2$ norm above is 
                \begin{equation}\label{E:prodfreeprops}
                    e^{it\partial_x^2}f(x) \bbar{e^{it\partial_x^2} g(x)}= L^{-1}\sum_{n,k}a_n \bbar{b_k}e^{\frac{2\pi i (n-k)x}{L}} e^{-\frac{4\pi^2 i (n^2-k^2)t}{L^2}}.
                \end{equation}
            By assumption, the Fourier support is separated by $M_0$. Therefore, the substitution of the summation index $k = n-m$ yields 
                \begin{equation*}
                    \eqref{E:prodfreeprops} = L^{-1}\sum_{|m| \geq M_0} e^{\frac{2\pi i m x}{L}} e^{\frac{4\pi^2 i m^2 t}{L}} \left(\sum_n a_n \bbar{b_{n-m}}e^{\frac{-4\pi^2 i (2nm )t}{L^2}}\right).
                \end{equation*}
                From now on, by exchanging the roles of $f$ and $g$, we may assume that $m$ is nonnegative.
                Now substituting this expression into the squared $L_{t,x}^2(I \times \T_L)$ norm, expanding and using orthogonality over $\T_L$, we have
                \begin{align}
                    \int_0^T \int_{\T_L} |e^{it\partial_x^2} f \bbar{e^{it\partial_x^2}g}|^2 \dx \dt &= L^{-1}\sum_{m \geq M_0} \int_0^T \left| \sum_{n} a_n \bbar{b_{n-m}} e^{\frac{4\pi^2i 2mn t}{L^2}} \right|^2 \dt \\
                    &=  L^{-1}\sum_{m\geq M_0} \frac{1}{m}\int_0^{mT}\left| \sum_n a_n \bbar{b_{n-m}} e^{\frac{4\pi^2i 2n\tau}{L^2}} \right|^2 
                    \dtau .
                    \label{E:Ltx2}
                \end{align}
                If $M_0 \geq \frac{L^2}{T}$, then we can continue from \eqref{E:Ltx2} by 
                \begin{align*}
                    \rm{RHS}\eqref{E:Ltx2} &\leq L^{-1}\sum_{m \geq \frac{L^2}{T}} \frac{1}{m} \int_0^{mT} \left| \sum_{n} a_n \bbar{b_{n-m}} e^{\frac{4\pi^2 i 2n\tau}{L^2}}\right|^2 \dtau \\
                    &\leq L^{-1}\sum_{\nu = 1}^\infty \sum_{\frac{\nu L^2}{T} \leq m \leq \frac{(\nu + 1)L^2}{T}}\frac{1}{m} \int_0^{L^2+ \nu L^2} \left|\sum_{n}a_n \bbar{b_{n-m}} e^{\frac{4\pi^2 i 2n\tau}{L^2}} \right|^2 \dtau.
                \end{align*} 
 Let $p_L:=\frac{L^2}{4\pi}$. By $p_L$-periodicity in time of the function under the integral, and since $(\nu+1)L^2=4\pi(\nu+1)p_L$, we can bound the integral over $[0, (\nu +1)L^2]$ by $\lceil 4\pi (\nu+1)\rceil\lesssim (\nu+1)$ times the integral from $0$ to $p_L$. Next note that the condition on the sum implies that $\tfrac{\nu+1}{m} \leq \tfrac{\nu+1}{\nu} \tfrac{T}{L^2}$. Continuing from above, we have 
                \begin{align*}
                    &
                        L^{-1}\sum_{\nu = 1}^\infty \sum_{\frac{\nu L^2}{T} \leq m \leq \frac{(\nu + 1)L^2}{T}}\frac{1}{m} \int_0^{L^2+ \nu L^2} \left|\sum_{n}a_n \bbar{b_{n-m}} e^{\frac{4\pi^2 i 2n\tau}{L^2}} \right|^2 \dtau  \\
                     \lesssim {}&L^{-1}\sum_{\nu = 1}^\infty\sum_{\frac{\nu L^2}{T} \leq m \leq \frac{(\nu + 1)L^2}{T}}\frac{(\nu + 1)}{m} \int_0^{p_L} \left|\sum_{n}a_n \bbar{b_{n-m}} e^{\frac{4\pi^2i 2n\tau}{L^2}} \right|^2 \dtau    
                    \\
                    \lesssim{}& \frac{T}{L^3} \sum_{\nu = 1}^\infty \sum_{\frac{\nu L^2}{T} \leq m \leq \frac{(\nu+1)L^2}{T}}\int_0^{p_L} \left|\sum_{n}a_n \bbar{b_{n-m}} e^{\frac{4\pi^2i 2n\tau}{L^2}} \right|^2 \dtau.
                    \end{align*}
                    By expanding the modulus squared and again using orthogonality of the complex exponentials on $[0, p_L]$, we get a bound of
                     \begin{align*} {}&\frac{T}{L^3} \sum_{ m=1}^\infty p_L \sum_{n}|a_n|^2|b_{n-m}|^2 
                    \lesssim{} \frac{T}{L}\|f\|_{L^2}^2 \|g\|_{L^2}^2.
                \end{align*}
                
If $M_0<L^2/T$, we have to add the term
\begin{align*}
     &L^{-1}\sum_{M_0\leq m < \frac{L^2}{T}} \frac{1}{m} \int_0^{mT} \left| \sum_{n} a_n \bbar{b_{n-m}} e^{\frac{4\pi^2i 2n\tau}{L^2}}\right|^2 \dtau \\
     & \hspace{1cm} \lesssim{} L^{-1}\sum_{M_0\leq m < \frac{L^2}{T}} \frac{1}{m} \int_0^{p_L} \left| \sum_{n} a_n \bbar{b_{n-m}} e^{\frac{4\pi^2 i 2n\tau}{L^2}}\right|^2 \dtau \\
     & \hspace{2cm} \lesssim{} \frac{L}{M_0} \|f\|_{L^2}^2 \|g\|_{L^2}^2,
     \end{align*} 
     again using orthogonality in $\tau$. Note that $T/L\leq L/M_0$ in this case.

     Next, for any threshold $k> M_0$ (chosen below), we also have the estimate
     \begin{align*}
     &L^{-1}\sum_{M_0\leq m < k} \frac{1}{m} \int_0^{mT} \left| \sum_{n} a_n \bbar{b_{n-m}} e^{\frac{4\pi^2 i 2n\tau}{L^2}}\right|^2 \dtau \\
    & \hspace{1cm} \leq{} \frac{T}{L}\sum_{M_0\leq m <k}  \left( \sum_n |a_n| |b_{n-m}| \right)^2 \\
      & \hspace{2cm} \lesssim{} \frac{kT}{L}\|f\|_{L^2}^2 \|g\|_{L^2}^2,
     \end{align*} 
     where we have used Cauchy-Schwarz. Therefore, if $M_0< L/T^{\frac12}$, we choose $k\sim L/T^{\frac12}$ so that the last two bounds are comparable and the claim is proved in the third case. 
     
     Note that if $M_0=0$, i.e.\ there is no separation of the supports, the third bound still holds.
            \end{proof}

We remark that due to the use of the triangle inequality on the Fourier coefficients in the above proof, the Theorem generalizes as follows:

\begin{cor}\label{cor:bil-op}
 For a bounded symbol $\sigma:\Z^2\to \C$ let
 $$B_\sigma(f,g)(t,x):=L^{-1}\sum_{m,n}\sigma(n,n-m) \widehat{f}(n) \bbar{\widehat{g}(n-m)}e^{\frac{2\pi i m x}{L}} e^{-\frac{4\pi^2 i (n^2-(n-m)^2)t}{L^2}}.$$
    Then, under the assumptions of Theorem \ref{T:MV-Ltorus}, we have
    \begin{equation*}
                    \left\| B_\sigma(f,g) \right\|_{L_{t,x}^2 (I \times \T_L)} \lesssim (\Delta(M_0, T))^{\frac{1}{2}}\|\sigma\|_{\ell^\infty} \|f\|_{L^2}\|g\|_{L^2},
                \end{equation*}
                with the same implicit constant and $\Delta(M_0, T)$ as in Theorem \ref{T:MV-Ltorus}. Moreover, the separation hypothesis may be dropped if $\sigma(n,n-m)=0$ whenever $|m|<M_0$.
\end{cor}
Note that \eqref{E:prodfreeprops} implies that $B_1(f,g)(t,x)=e^{it\partial_x^2}f(x) \bbar{e^{it\partial_x^2} g(x)}$.

          Below, we will now take $T=1$ to match the corresponding integration interval we have used in the introduction.  Our arguments do not change terribly if we change the width of the time interval we integrate over.

           We can state the following corollary, the proof of which follows similarly as in the proof of Corollary $1$ in \cite{HL09}.
            \begin{cor}\label{C:multlinearTL}
                Suppose that $f_1,\ldots,f_4 \in L^2(\T_L)$ such that there exist $1\leq i,j\leq 4$ with
                \[
                  \delta = \rm{dist}(\supp{\hat{f_i}}, \supp{\hat{f_j}}) > L^\gamma, 1< \gamma \leq 2 ,
                \]
                then the following estimate holds:
                \begin{equation*}
                   |\cQ(f_1, f_2, f_3, f_4)|
                   \lesssim  L^{\frac{1-\gamma}{2}}  \prod_{j = 1}^4 \|f_j\|_{L^2(\T_L)}.
                \end{equation*}
                \begin{proof}
                    Using the structure of $\cQ$ and Cauchy-Schwarz, we have 
                    \begin{equation*}
                        |\cQ(f_1, f_2, f_3, f_4)|\leq \|T_r f_i T_r f_j\|_{L_{t, x}^2} \|T_r f_\ell T_r f_m\|_{L_{t, x}^2}
                    \end{equation*}
                    For the term involving $\ell$ and $m$, we obtain the crude bound $C$. For the term involving functions $i$ and $j$, we see that \Cref{T:MV-Ltorus} implies a gain of a factor of $L^{\frac{1-\gamma}{2}}$.
                \end{proof}
            \end{cor}
  
            \begin{lem}[Fourier space quasilocality]\label{E:quasiloc}
            Let $s>0$ and $i \in \{1,2,3,4\}$.
                If $\supp{\hat{f_i}} \subset \{|k| \geq 3s\}$ and $\supp{\hat{f_j}} \subset \{|k| < s\}$ for all $j \neq i$ then 
                \begin{equation*}
                    \cQ(f_1, f_2, f_3, f_4) \equiv 0
                \end{equation*}
            \end{lem}
            \begin{proof}
               We follow the same strategy as in \cite{HL09}, Lemma $6$.  Re-writing the problem in Fourier space as in \cref{E:qdefFourier}, we then use the given support conditions to observe that all the frequency interactions are $0$. 
            \end{proof}
            Next, following the argument in \cite{HL09}, we derive a `self-consistency' estimate which will be used in a bootstrap argument. Closing this bootstrap will allow us to derive the desired polynomial decay bounds. Before we state the estimate, we define the following quantities (for any function $\vphi \in L^2(\T_L)$):
            \begin{dfn}
                Let $\vphi \in L^2(\T_L)$. Then we set (for any $s > 0$)
                \begin{align*}
                    \hat{\vphi_\lo}(n) &= \hat{\vphi}(n)\bone_{[s]}(n) , \\
                    \hat{\vphi_\hi}(n) &= \hat{\vphi}(n)(1-\bone_{[ s]}(n)).
                \end{align*}
                Here, $[s] = \{- \lceil s \rceil, \ldots, 0, \ldots, \lceil s \rceil\} \subset \Z$, as defined in Section \ref{S:notation}. We then define the $\ell^2$-tail of $\vphi$ by 
                \begin{equation*}
                    \beta_s(\vphi) := \| \vphi_\hi\|_{L^2}.
                \end{equation*}
                If the argument of $\beta_s$ is omitted, its argument will be clear from context.
            \end{dfn}
            We now state the self-consistency estimate:
            \begin{lem}[Self-consistency estimate]\label{L:selfcons}
                Let $f \in L^2(\T_L)$ be a mass-$1$ weak solution to 
                \eqref{E:0av-dmnls-periodic}, i.e.
                \[
                  \omega \langle g,f\rangle  = \cQ(g,f, f,f) \text{ for all }g \in L^2(\T_L),
                \]
                for some $\omega>0$. Define $\beta_s$ to be the Fourier space $\ell_2$-tail of $f$, then for any 
                 $1 < \gamma \leq 2$, any $L > 1$, and any $s > \tfrac12 L^\gamma$ there exists a constant $K > 0$ independent of $L,\gamma,$ and $s$, such that
                \begin{equation*}
                    \omega \beta_{3s} \leq K( \beta_s^3 + L^{\frac{1-\gamma}{2}} \beta_0^2 \beta_s).
                \end{equation*}
            \end{lem}
            \begin{proof}
                Begin by setting 
                \begin{align*}
                    \hat{f_\lo}(n) &= \hat{f}(n)\bone_{[s]}(n), \\
                    \hat{f_\hi}(n) &= \hat{f}(n)(1-\bone_{[s]})(n).    
                \end{align*}
                Then by Fourier inversion we have $f = f_\lo + f_\hi$. Further, by duality, 
                \begin{equation*}
                    \beta_{3s} = \sup_{\substack{\supp{\hat{g}} \subset \Z \setminus [3s]\\ \|g\|_{L^2}=1}} |\jbrak{g, f}|.
                \end{equation*}

                Since $f$ is a weak solution to the dispersion management equation, this tells us we should estimate 
                    $|\cQ(g, f,f,f)|$,
                where $g$ obeys the support and normalization conditions above.

                We 
                split each copy of $f$ as $f = f_\lo + f_\hi$
                and use multilinearity to obtain
                \begin{multline}
                    \cQ(g, f,f,f) = \cQ (g, f_\lo, f_\lo, f_\lo) + \cQ (g, f_\hi, f_\hi, f_\hi) \\+ 
                    \cQ(g, f, f_\hi, f_\lo) + \cQ(g, f_\hi, f_\lo, f) + \cQ(g, f_\lo, f, f_\hi). 
                \end{multline}
                The first term is zero by the quasilocality.
               
                Since $\cQ$ is a bounded map $L^2(\T_L)^{\otimes 4} \to \R$, the second term is controlled by 
                \begin{equation*}
                   | \cQ(g, f_\hi, f_\hi, f_\hi) |\lesssim \|g\|_{L^2} \beta_s^3.
                \end{equation*}
                The remaining $3$ terms obey frequency separation conditions that allow us to apply \Cref{C:multlinearTL} and estimate them by 
                \begin{equation*}
                  |  \cQ(g, f_1,f_2,f_3)| \lesssim  L^{\frac{1-\gamma}{2}}  \|g\|_{L^2} \beta_0^2 \beta_s, \quad \{f_1,f_2,f_3\}=\{f,f_{\lo},f_{\hi}\}.
                \end{equation*}
                
                Adding up, we obtain
                \begin{equation*}
                    \omega \beta_{3s} \lesssim \beta_s^3 + L^{\frac{1-\gamma}{2}} \beta_0^2 \beta_s
                \end{equation*}
as claimed.
            \end{proof}

            \begin{cor}
            \label{cor:HL} Let $ \gamma \in (1, 2]$ and $K$ be as in Lemma \ref{L:selfcons}.
                 Let $f$ be a solution as in Lemma \ref{L:selfcons} with $\omega=\omega_1(L)$
                 . Suppose that $L>1$ is large enough so that \begin{equation}\label{eq:cond-L}
                 2K L^{\frac{1-\gamma}{2}}<3^{-\frac14} \omega .\end{equation}
                 Then, there exists a constant $s_0 (f,L,\gamma)\geq L^\gamma$ sufficiently large, that for all $$s > \tilde s_0 := 3 s_0 L^{\gamma-1},$$ the following estimate for $\beta_s$ holds:
                \begin{equation}
                    \beta_s \lesssim_{s_0,\alpha,L,\gamma} s^{-\alpha},
                \end{equation}
                where $\alpha= \log_3(\frac{\omega}{2K} L^{\frac{\gamma-1}{2}})>0$.
            \end{cor} 

            \begin{proof}
                By \Cref{L:selfcons}
                and our hypothesis $\|f\|_{L^2}=1$, we know there exists a constant $K >0$ independent of $L,\gamma$, so that 
                \begin{equation}
                    \beta_{3s} \leq \frac{K}{\omega}\left\{\beta_s^2 + L^{\frac{1-\gamma}{2}} \right\}\beta_s \text{ for all } s>\frac12 L^\gamma.
                \end{equation}
                Using the fact that $s \mapsto \beta_s$ monotonically decreases to $0$, and that $\omega_1(L)$ is uniformly bounded, by our assumption \eqref{eq:cond-L} on $L > 1$ we can choose $s_0 = s_0(f,L,\gamma)\geq L^\gamma$ large enough to obtain 
                \begin{equation}\label{E:aprioriest}
                    \frac{K}{\omega}\left\{ \beta_{s_0}^2 + L^{\frac{1-\gamma}{2}}  \right\} \leq 3^{-\frac{1}{4}}.
                \end{equation}
                This yields the \textit{a priori} estimate 
                \begin{equation*}
                    \beta_{3s} \leq 3^{-\frac{1}{4}} \beta_s \qtq{for all} s \geq s_0(L,\gamma).
                \end{equation*}
                Using this estimate, we have an algebraically decreasing bound  
                \begin{align*}
                \beta_s & \leq \beta_{s_0} (3 s_0)^{\frac14}s^{-\frac14}  \leq (3 s_0)^{\frac14}   s^{-\frac14},
                \end{align*}
                 because $\beta_{s_0}\leq 1$. With $C(s_0) =(3s_0)^{\frac14}$ we then have the bound 
            \begin{equation}\label{E:firstbootstrap}
                \beta_s \leq C(s_0) s^{-\frac{1}{4}} \qtq{for} s \geq s_0.
            \end{equation}
             We now feed this estimate back into \eqref{E:aprioriest} and bootstrap. We readily see that \eqref{E:aprioriest} and \eqref{E:firstbootstrap} together imply 
            \begin{equation}
                \beta_{3s} \leq \frac{K}{\omega} \left\{C(s_0)^2 s^{-\frac12} + L^{\frac{1-\gamma}2}\right\}\beta_s \qtq{for} s \geq s_0.
            \end{equation}
            In particular, by matching powers inside the curly braces, we can guarantee the following bound: 
            \begin{equation}
                \beta_{3s} \leq 3^{-\log_3(\frac{\omega}{2K}L^{\frac{\gamma-1}2})}\beta_s \qtq{for} s \geq \tilde{s_0} := C(s_0)^4 L^{\gamma-1}.
            \end{equation}
            As in \cite{HL09}, we introduce the quantity $\lambda(t) = \log_3(\beta(3^t))$ and set $\alpha \defe \log_3(\frac{\omega}{2K} L^{\frac{\gamma-1}{2}})$. The \textit{a priori} estimate from earlier implies that 
            \begin{equation}
                \lambda(t+1) -\lambda(t) \leq -\alpha \qtq{for} t\geq \tilde{t_0} = \log_3(\tilde{s_0}
                ).
            \end{equation}
            Now define $\tilde{\lambda}(t) = \lambda(t) + t\alpha$. The estimate on $\lambda$ implies that 
            \begin{equation}
                \tilde{\lambda}(t+1) - \tilde{\lambda}(t) \leq 0.
            \end{equation}
             In particular, we have 
            \begin{equation}
                \tilde{\lambda}(t) \leq \sup_{s \in [t_0, t_0+1]}\tilde{\lambda}(s) \leq \lambda(t_0) + (t_0+1)\alpha.
            \end{equation}
            Tacing the definitions back to $\beta(s)$, a bit of simple algebra reveals the estimate 
            \begin{equation}
                \beta(s) \lesssim_{s_0, \alpha, L, \gamma} s^{-\alpha} \qtq{for all} s\geq \tilde{s_0}.
            \end{equation}
            \end{proof}      
 Therefore, polynomial decay as stated in Theorem \ref{thm:poly-decay} holds.
 
 \begin{remark}\label{rem:om}
 The condition \eqref{eq:cond-L} is satisfied for sufficiently large $L$ 
 because in Section \ref{sec:largeL} we will show that $\liminf_{L\to \infty} \omega_1(L)>0$.

 \end{remark}
        \section{Exponential Decay}
        \label{sec:exp}

       Having established that some decay must occur in the Fourier series coefficients of an optimizer, we now argue analogously to \cite{EHL11} in order to prove that in fact the decay must be exponential in nature.  We remark that this is sufficient to prove that periodic solutions to \eqref{E:0av-dmnls-periodic} are analytic.  To this end, define 
        \begin{equation}
            F_{\mu, \eps}(\cdot) = \mu \frac{|\cdot|}{1+\eps|\cdot|}, \qtq{for} \mu, \eps > 0. 
        \end{equation}
        and the modified functional 
        \begin{equation}\label{E:Qmueps}
        \tilde{\cQ}_{\mu, \eps}(f_1, f_2, f_3, f_4) \defe  \cQ\left(e^{F_{\mu, \eps}(P)}f_1, e^{-F_{\mu, \eps}(P)}f_2, e^{-F_{\mu, \eps}(P)}f_3, e^{-F_{\mu, \eps}(P)}f_4\right)
        \end{equation}
        where $e^{F_{\mu, \eps}(P)}$ denotes the Fourier multiplier defined by $$\widehat{e^{F_{\mu, \eps}(P)}f}(k)=e^{F_{\mu, \eps}(k)}\widehat{f}(k).$$
        The crucial ingredients in the proof are 
        \Cref{P:sepsupportsbound} and a bootstrap argument, which are analogous to \cite[Theorems 2.2 and 2.3 and Lemma 3.1]{EHL11}. 

We aim to prove boundedness and, in case of separated supports, exploit the gain from the bilinear estimate in Theorem \ref{T:MV-Ltorus} and Corollary \ref{cor:bil-op}.
Since $e^{\pm F_{\mu, \eps}(P)}$ preserves Fourier supports, Lemma \ref{E:quasiloc} applies verbatim to $\tilde{\cQ}_{\mu, \eps}$
. For convenience, we set
        \[
\Delta_L(\delta):=\Delta(\delta,1)=\begin{cases} 1 & \delta\leq L\\
\frac{L}{\delta}& L<\delta\leq L^2\\
\frac{1}{L} & \delta>L^2,
\end{cases}
        \]
        for $\delta\geq 0$.
        \begin{pro}[Boundedness and gain from separated supports]\label{P:sepsupportsbound}
        Let $\mu,\eps\geq 0$, and $f_1,\ldots,f_4\in L^2(\T_L)$.
        Let either $\delta=0$ or let
            \begin{equation}
            \label{eqn:deltaassump}
                \delta := \rm{dist}(\supp{\hat{f_\ell}}, \supp{\hat{f_k}}) >0
            \end{equation}
            for some $\ell, k \in \{1,2,3,4\}$. Then, 
            \begin{equation}
            \label{eqn:freqsep}
              |\tilde{\cQ}_{\mu, \eps}(f_1, f_2, f_3, f_4)|
              \leq c_0 \big(\Delta_{L}(\delta)\big)^\frac12 \prod_{j = 1}^4 \|f_j\|_{L^2},
            \end{equation}
            where the constant $c_0>0$ is independent of $\mu, \eps \geq 0$ and of the $f_j$.
        \end{pro}
For notational brevity, let
        \begin{align*}
            \cD =& \{ \bk \in \Z^4 \mid k_1 -k_2 + k_3 - k_4 = 0\}.
        \end{align*}
        We again follow \cite{EHL11} and motivated by \eqref{eq:fourier-eq}, for a multiplier 
        $ M : \Z^4 \to \R  $
        we define
        \begin{align}
        \label{KMeqn}
            & \K^{\cD}_{M}(f_1, f_2, f_3, f_4) \defe \\
            & \hspace{.5cm} L^{-1}\int_0^1 \sum_{\bk \in \cD} e^{\frac{4\pi^2i(k_1^2-k_2^2+k_3^2-k_4^2)r}{L^2}}M(\mathbf{k})\left[\prod_{j=1}^4 \mathscr{C}^j\hat{f_j}(k_j)\right] \dr  ,\notag
        \end{align}
        and similarly we define $\K^{\cD'}_{M}$ for any subset
        $\cD'$. 
        Recall from Section \ref{S:notation} that the phase $\Phi(\bk)=2(k_1-k_4)(k_4-k_3)$ degenerates exactly on $\cD_{14}\cup\cD_{34}$, where
        $\cD_{ij}=\{\bk \in \cD : k_i=k_j\}$.
        With this notation,
\begin{equation}\label{E:qk}
            \tilde{\cQ}_{\mu, \eps} =\K^{\cD}_{M_{\mu,\eps}}
        \end{equation}
        where 
        \begin{equation}\label{E:Qmueps_multiplier}
              M_{\mu, \eps}(\mathbf{k}) \defe e^{F_{\mu,\eps}(k_1)-F_{\mu,\eps}(k_2)-F_{\mu,\eps}(k_3)-F_{\mu,\eps}(k_4)}.  
        \end{equation}
        \begin{lem}[Multiplier bounds]\label{L:Qmuepsmultbounds}
            Let $M_{\mu, \eps}(\bk)$ be the multiplier given in \eqref{E:Qmueps_multiplier}. Then we have that $\|M_{\mu, \eps}(\bk)\|_{\ell^\infty(\cD)} \leq 1$. 
        \end{lem}
        \begin{proof}
           To see the above claim, we use that for any $k_1, k_2 \in \mathbb{R}$, we have
            \begin{equation}
                \label{eqn:tri}
                 \frac{ |k_1 - k_2| }{1 + \eps  |k_1 - k_2| } \leq     \frac{ |k_1| }{1 + \eps  |k_1 | } +     \frac{ |k_2| }{1 + \eps  |k_2| }
                 \end{equation}
            and iterate.  In particular, using \eqref{eqn:tri}, we have for $k_1,k_2,k_3 \in \mathbb{R}$
            \begin{align}
                \label{eqn:quad}
            \frac{ |k_1 - k_2+k_3| }{1 + \eps  |k_1 - k_2 + k_3| } & \leq     \frac{ |k_1-k_2| }{1 + \eps  |k_1 - k_2 | } +     \frac{ |k_3| }{1 + \eps  |k_3| } \\ 
            & \leq  \frac{ |k_1| }{1 + \eps  |k_1 | } + \frac{ |k_2| }{1 + \eps  |k_2 | } +     \frac{ |k_3| }{1 + \eps  |k_3| }. \notag
            \end{align}
            
            Inequality \eqref{eqn:tri} is easily seen to be equivalent by finding a common denominator and recognizing that an equivalent formulation is
            \[
                |k_1 - k_2|  (1 + \eps |k_1|) (1 + \eps |k_2| )  \leq   (|k_1| + |k_2 | + 2 \eps |k_1| |k_2|)( 1 + \eps  |k_1 - k_2| ),
            \]
        which follows easily from the triangle inequality. 
        
        Hence, given $k_1 = k_2 - k_3 + k_4$, we observe that
        \[
        e^{F_{\mu,\epsilon}(k_1) - F_{\mu,\epsilon}(k_2) - F_{\mu,\epsilon}(k_3) -F_{\mu,\epsilon}( k_4)} \leq 1,
        \]
        which completes the proof.
    \end{proof}

            \begin{proof}[Proof of Proposition \ref{P:sepsupportsbound}]
By \eqref{E:qk}, it suffices to prove the bound for $\K^{\cD}_{M_{\mu,\eps}}$.

From now on we suppress the subscript $\mu,\eps$ here and write $M=M_{\mu,\eps}$ and $F=F_{\mu,\eps}$.
From \eqref{eqn:tri} we deduce that the even and non-negative function $F$ is subadditive, therefore
\begin{align*}
0<\sigma_1(k_1,k_2)&:=e^{F(k_1)-F(k_2)-F(k_1-k_2)}\leq 1, \\0<\sigma_2(k_3,k_4)&:=e^{F(k_4-k_3)-F(k_3)-F(k_4)}\leq 1.
\end{align*}
On $\cD$, using $k_1-k_2=k_4-k_3$, we therefore have
\[M(\bk)=\sigma_1(k_1,k_2)\sigma_2(k_3,k_4),\]
each separately bounded by $1$. Similarly, $M(\bk)=\tilde{\sigma}_1(k_1,k_4)\tilde{\sigma}_2(k_2,k_3)$, for
\[
\tilde{\sigma}_1(k_1,k_4):=e^{F(k_1)-F(k_4)-F(k_1-k_4)}, \; \tilde{\sigma}_2(k_2,k_3):=e^{F(k_2-k_3)-F(k_2)-F(k_3)},
\]
both of which are non-negative and bounded by $1$.

By Cauchy-Schwarz, we obtain
\begin{equation}\label{eq:bil-split}
        |\K^{\cD}_{M}(f_1,f_2,f_3,f_4)|\lesssim \|B_{\sigma_1}(f_1,f_2)\|_{L^2_{r,x}}  \|B_{\sigma_2}(f_3,f_4)\|_{L^2_{r,x}}
\end{equation}
Now, Corollary \ref{cor:bil-op} proves the claim in the case $\delta\leq L$, since $|\sigma_1|,|\sigma_2|\leq 1$.

For the stronger bounds under the separation hypothesis, we need to distinguish cases:

In Case i) where $\{k,\ell\}=\{1,2\}$ or $\{k,\ell\}=\{3,4\}$, estimate \eqref{eq:bil-split} and Corollary \ref{cor:bil-op} imply the claim, too.

In Case ii) where $\{k,\ell\}=\{1,4\}$ or $\{k,\ell\}=\{2,3\}$, we similarly use Cauchy-Schwarz
\begin{equation}\label{eq:bil-split2}
        |\K^{\cD}_{M}(f_1,f_2,f_3,f_4)|\lesssim \|B_{\tilde{\sigma}_1}(f_1,f_4)\|_{L^2_{r,x}}  \|B_{\tilde{\sigma}_2}(f_2,f_3)\|_{L^2_{r,x}}
\end{equation}
and then Corollary \ref{cor:bil-op}.

The remaining cases are iii) a) $\{k,\ell\}=\{1,3\}$ or iii) b) $\{k,\ell\}=\{2,4\}$, where we need to do a different splitting. In Case iii) a) we observe that $\delta\leq |(k_1-k_2)+(k_2-k_3)|$ within the total Fourier support, and in Case b), because of $k_1-k_4=k_2-k_3$, we observe that $\delta\leq |(k_1-k_2)-(k_2-k_3)|$. Now, both cases we therefore either have  $|k_1-k_2|\geq \delta/2$ or $|k_1-k_2|< \delta/2$ where $|k_2-k_3|\geq \delta/2$ holds.

Fix a bump function $\rho$, i.e. an even $\rho\in C^\infty(\R)$, $0\leq \rho\leq 1$, supported in $(-2,2)$ and equal to $1$ on $[-1,1]$, and define $\rho_\delta(\tau):=\rho(4\tau /\delta)$.
Then, we decompose $M(\bk)=M_1(\bk)+M_2(\bk)$, where $$ M_1(\bk)=M(\bk)(1-\rho_\delta(k_1-k_2)), \; M_2(\bk):=M(\bk)\rho_\delta(k_1-k_2).$$ For the first term, since the symbol $\sigma_1(k_1,k_2)(1-\rho_\delta(k_1-k_2))$ is supported in $|k_1-k_2|\ge \delta/4$,
the above  bound \eqref{eq:bil-split} and  Corollary \ref{cor:bil-op}  imply the claim. Within the support of the second term we have $|k_1-k_2|< \delta/2$, which implies $|k_2-k_3|\geq \delta/2$ and we would like to use \eqref{eq:bil-split2} and Corollary \ref{cor:bil-op}, but the multiplier $M_2$ with the cutoff cannot be factorized accordingly. Therefore, we first write
\[\rho_\delta(k_1-k_2)=\int \widehat{\rho_\delta}(\tau) e^{2\pi i \tau k_1} e^{-2\pi i \tau k_2}d\tau.\]
Then, with $M_2^\sharp(\bk):=\mathbf{1}_{|k_2-k_3|\geq \delta/2} M(\bk)$, and $\tilde{\sigma}_2^\sharp(k_2,k_3):=\mathbf{1}_{|k_2-k_3|\geq \delta/2}\tilde{\sigma}_2(k_2,k_3)$,
\begin{align*}
  |\K^{\cD}_{M_2}(f_1,f_2,f_3,f_4)|\leq  \int  |\widehat{\rho_\delta}(\tau)| |\K^{\cD}_{M_2^\sharp }(f_1(\cdot -\tau L),f_2(\cdot -\tau L) ,f_3,f_4)|d\tau \\\leq
  \|\widehat{\rho}\|_{L^1}\sup_\tau  \|B_{\tilde{\sigma}_1}(f_1(\cdot -\tau L),f_4)\|_{L^2_{r,x}}  \|B_{\tilde{\sigma}_2^\sharp}(f_2(\cdot -\tau L),f_3)\|_{L^2_{r,x}},
\end{align*}
where we have used that $\|\widehat{\rho_\delta}\|_{L^1}=\|\widehat{\rho}\|_{L^1}$.
Now, since translation affects neither the Fourier support nor the $L^2$ norm, Corollary \ref{cor:bil-op} implies the claim.
 \end{proof}

\begin{remark}
    The degenerate part of the phase does not cost anything, since in the degenerate sets $\cD_{14}$, $\cD_{34}$ and their intersection the integration in $r$ is trivial, the multiplier is bounded by $1$, and two applications of Cauchy-Schwarz give the upper bound $L^{-1}\prod_{j = 1}^4 \|f_j\|_{L^2}$, which is better than \eqref{eqn:freqsep}. So, no renormalization (such as Wick ordering) is needed in the argument.
\end{remark}
 
Next, we set up the bootstrap argument in \cite{EHL11} by first establishing an \textit{a priori} estimate. Before we state the lemma, we define
\begin{equation*}
\|f\|_{\mu, \eps} = \|e^{F_{\mu, \eps}(P)}f\|_{L^2(\T_L)}. 
\end{equation*}

For $\tau > 3$ define the following quantities:
        \begin{align*}
        f_{\ll} &= \F^{-1} \bone_{[\tau/3]}(k) \hat{f}(k),\ \ f_< = \F^{-1} \bone_{[\tau]}(k) \hat{f}(k), \\
        f_> &= \F^{-1} (1-\bone_{[\tau]}(k))\hat{f}(k), \ \   f_\sim = f_< - f_\ll,
    \end{align*}
    where $[\cdot ]\subset \Z$ is the symmetric set of integers defined in Section \ref{S:notation}. Note that $\widehat{f}_>$ and $\widehat{f}_\ll$ are separated by at least $\frac23 \tau$.
    
\begin{lem}[Consistency estimate]\label{L:selfconsest}
    Let $f\in L^2(\T_L)$ with $\|f\|_{L^2}=1$ be a weak solution, for some $\omega>0$, of the dispersion-managed equation
    \[\omega \langle \phi,f\rangle= \cQ(\phi,f,f,f) \quad \forall \phi \in L^2(\T_L),\]
    Then, for all $\tau>3$ and all $\mu,\eps\geq 0$ we have the following self-consistency estimate
    \begin{equation}
    \omega \|f_>\|_{\mu, \eps} \lesssim \begin{multlined}[t]
            \| f_> \|_{\mu, \eps}^3 +  e^{\mu \tau}\| f_> \|_{\mu, \eps}^2 \\+ e^{2\mu\tau} \|f_> \|_{\mu, \eps}\left(\big(\Delta_L(\tau)\big)^{\frac12}+ \|f_\sim\|_{L^2}\right) \\
            + e^{3\mu\tau} \left( \big(\Delta_L(\tau)\big)^{\frac12}+ \|f_\sim\|_{L^2}\right).
        \end{multlined}
    \end{equation}
\end{lem}

\begin{proof}
The proof relies on the estimate from Proposition \ref{P:sepsupportsbound}.
Test the equation by $\phi = e^{2F_{\mu, \eps}(P)}f_>$. Then,
    \begin{equation}
        \omega \|f_>\|_{\mu, \eps}^2 = \tilde{\cQ}_{\mu, \eps}\left(e^{F_{\mu, \eps}(P)}f_>, e^{F_{\mu, \eps}(P)}f, e^{F_{\mu, \eps}(P)}f, e^{F_{\mu, \eps}(P)}f\right).
    \end{equation}
    Setting $h = e^{F_{\mu, \eps}(P)}f$, the above display reads 
    \begin{equation}
        \omega \|h_>\|_{L^2}^2 = \tilde{\cQ}_{\mu, \eps}(h_>, h, h, h).
    \end{equation}
    We now split $h = h_< + h_>$, and use the multilinearity to split the right-hand side up into 
    \begin{align*}
        \tilde{\cQ}_{\mu, \eps}(h_>, h, h,h) &= \tilde{\cQ}_{\mu, \eps}(h_>, h_> + h_< , h_> + h_< , h_> + h_<) \\
        &= \tilde{\cQ}_{\mu, \eps}(h_>, h_>, h_>, h_>) + \tilde{\cQ}_{\mu, \eps}(h_>, \text{$3$ low}) \\
        &+ \tilde{\cQ}_{\mu, \eps}(h_>, \text{$2$ high, $1$ low}) + \tilde{\cQ}_{\mu, \eps}(h_>, \text{$1$ high, $2$ low}),
    \end{align*}
    where we have identified extra terms that only differ by a permutation of the last three entries, as their estimates are identical. Every term below is estimated by Proposition \ref{P:sepsupportsbound} as follows:
    when no separation is present, with $\delta=0$, and for the contributions  of $h_\ll$, with $\delta=\frac23 \tau$.
    For the first term, taking $\delta=0$, we obtain
    \begin{equation}
        |\tilde{\cQ}_{\mu, \eps}(h_>, h_>, h_>, h_>)| \lesssim \|h_> \|_{L^2}^4.
    \end{equation}
    We then bound 
    \begin{equation}
        |\tilde{\cQ}_{\mu, \eps}(h_>, \text{$2$ high, $1$ low})| \lesssim   \|h_> \|_{L^2}^3 \|h_< \|_{L^2}
    \end{equation}
    in the same way.

    For the remaining terms (which all contain at least one copy of $h_<$), we split $h_< = h_\ll + h_\sim$. Using multilinearity again, we bound the remaining $2$ kinds of terms as follows:

    \begin{align}
       &|\tilde{\cQ}_{\mu, \eps}(h_>, \text{$3$ low})| = |\tilde{\cQ}_{\mu, \eps}(h_>, \text{$2$ low}, h_\ll + h_\sim)| \nonumber\\
       &\lesssim |\tilde{\cQ}_{\mu, \eps}(h_>, \text{$2$ low}, h_\ll)| + |\tilde{\cQ}_{\mu,\eps}(h_>, \text{$2$ low}, h_\sim)| \nonumber \\
       &\lesssim \big(\Delta_L( \tau) \big)^{\frac12}\|h_> \|_{L^2} \|h_< \|_{L^2}^2 \|h_\ll\|_{L^2} + \|h_> \|_{L^2} \|h_<\|_{L^2}^2 \|h_\sim\|_{L^2},
    \end{align}
and 
    \begin{align}
     &   |\tilde{\cQ}_{\mu, \eps}(h_>, \text{$1$ high, $2$ low})| = |\tilde{\cQ}_{\mu, \eps}(h_>, \text{$1$ high, $1$ low}, h_\ll + h_\sim)| \nonumber \\
        &\lesssim |\tilde{\cQ}_{\mu, \eps}(h_>, \text{$1$ high, $1$ low}, h_\ll)| + |\tilde{\cQ}_{\mu, \eps}(h_>, \text{$1$ high, $1$ low}, h_\sim)|\nonumber \\
        &\lesssim  \big(\Delta_L(\tau) \big)^{\frac12}\|h_>\|_{L^2}^2 \|h_<\|_{L^2} \|h_\ll\|_{L^2} + \|h_> \|_{L^2}^2 \|h_<\|_{L^2}\|h_\sim\|_{L^2}.
    \end{align}
    Summing the estimates and dividing through by $\|h_>\|_{L^2}$, we arrive at
    \begin{align}
        \omega \|h_>\|_{L^2} \lesssim{} &
            \| h_> \|_{L^2}^3 + \| h_> \|_{L^2}^2 \|h_<\|_{L^2}+ \|h_>\|_{L^2}\|h_<\|_{L^2}\|h_\sim\|_{L^2}\nonumber\\ &+ \big(\Delta_L( \tau) \big)^{\frac12} \|h_>\|_{L^2} \|h_<\|_{L^2}\|h_\ll\|_{L^2} \nonumber\\
            &+ \big(\Delta_L( \tau) \big)^{\frac12}\|h_<\|_{L^2}^2\|h_\ll\|_{L^2}  + \|h_<\|_{L^2}^2\|h_\sim\|_{L^2}.
    \end{align}
    Now note that in Fourier space, $|\widehat{h_<}|$ and $|\widehat{h_\ll}|$ are bounded by $e^{\mu \tau} |\hat{f}|$ and $|\widehat{h_\sim}|$ is controlled by $e^{\mu \tau}|\hat{f_\sim}|$. By Plancherel's theorem, we can bound 
    \begin{equation*}
        \|h_<\|_{L^2} \leq e^{\mu\tau}\|\hat{f}\|_{\ell^2}, \quad 
        \|h_\ll\|_{L^2} \leq e^{\mu\tau}\|\hat{f}\|_{\ell^2}, \quad
        \|h_\sim\|_{L^2} \leq e^{\mu\tau}\|\hat{f_\sim}\|_{\ell^2}.
    \end{equation*}
    By Plancherel, and using the condition that $\|f\|_{L^2} = 1$, we arrive at 
    \begin{align}
    \omega \|h_>\|_{L^2} \lesssim{}&
            \| h_> \|_{L^2}^3 +  e^{\mu \tau}\| h_> \|_{L^2}^2 \nonumber \\+ & e^{2\mu\tau} \|h_> \|_{L^2}\left(\big(\Delta_L( \tau) \big)^{\frac12} + \|f_\sim\|_{L^2}\right) 
            + e^{3\mu\tau} \left(\big(\Delta_L( \tau) \big)^{\frac12} + \|f_\sim\|_{L^2}\right).
    \end{align}
    By recalling the definition of $\|f\|_{\mu, \eps}$, we conclude the claim.
\end{proof}

With this self-consistency lemma in hand, we can now prove the claim that the Fourier modes of optimizers necessarily decay exponentially. We introduce some notation first. Recall that $c_0>0$ denotes the unspecified constant in \eqref{eqn:freqsep}, which is independent of $\mu,\eps\geq 0$. For $\omega>0$ define the polynomial
$$G_\omega(\nu):=\frac{\omega}{2}\nu -c_0 \nu^2-c_0 \nu^3$$
and let $\nu_{\rm{\max}}>0$ be the maximizer of $G_\omega:[0,\infty) \to \R$ and define $g_0(\omega):=G_\omega(\frac{\nu_{\rm{\max}}}{2})$, which is a positive number.
Note that $G_\omega$ is the same polynomial as in \cite[eq. (3.5)]{EHL11}.

\begin{pro}[Exponential Decay of Fourier Modes]
Let $f$ be as in Lemma \ref{L:selfconsest}, i.e. a mass-$1$ weak solution of the dispersion managed equation, for some $\omega>0$.
Suppose that $L>0$ is large enough that
\begin{equation}\label{eq:cond-L-exp}
c_0L^{-\frac12}\leq \tfrac12 \min\big(
\tfrac{\omega}{2},g_0(\omega)\big).
\end{equation}
Then, there exists $\mu=\mu(L,\omega,f)$ so that the bound
    \begin{equation}
        |\hat{f}(k)|\lesssim e^{-\mu|k|} 
    \end{equation}
    holds for all $k \in \Z$. In particular, $f$ is analytic.
\end{pro}

\begin{proof}
    Having established Lemma \ref{L:selfconsest}, we may now follow the proof given in \cite{EHL11} to prove exponential decay in $k$.
    
    Given $\tau >3$ (which we will choose below), choose $\mu$ so that $e^{\mu \tau} = 2$. Now, by the self-consistency estimate above, with $\nu = \| f_> \|_{\mu,\eps}$, we obtain
    \begin{equation}
        \left( \omega - c_0 \big(\Delta_L(\tau)\big)^{\frac12} - c_0  \| f_{\sim} \|_{L^2} \right) \nu - c_0 \nu^2 - c_0 \nu^3 \leq c_0 \left( \big( \Delta_L(\tau)\big)^{\frac12} + \| f_{\sim} \|_{L^2} \right) .
    \end{equation}
    Now, we choose $\tau > 3$ large enough to guarantee 
    \begin{align}
        c_0\left(\big(\Delta_L(\tau)\big)^{\frac12} +  \|f_\sim\|_{L^2} \right) &\leq \min\left(\frac{\omega}{2},g_0(\omega)\right), 
        \label{eqn:a1} \\
        &\|f_>\|_{L^2} \leq \frac{\nu_{\rm{\max}}}{4} .
        \label{eqn:a2}
    \end{align}
    This is possible because $\Delta_L(\tau)=L^{-1}$ for $\tau> L^2$, so that our assumption \eqref{eq:cond-L-exp} on $L$ takes care of the first summand, while the other terms are handled by the fact that for a given $f \in L^2(\T_L)$ we have that
    $$\lim_{\tau \to \infty}\|f_\sim\|_{L^2}=\lim_{\tau \to \infty}\|f_>\|_{L^2}= 0.$$
    We now observe from the self-consistency estimate and \eqref{eqn:a1}-\eqref{eqn:a2} that 
    \begin{equation}\label{eq:g-b}
        G_\omega(\|f_>\|_{\mu, \eps}) \leq g_0(\omega)=G_\omega(\frac{\nu_{\rm{max}}}{2}).
    \end{equation}
    We now show that for the choice of $\mu$ above, we have that $e^{\mu |\cdot|}\hat{f} \in \ell^2$. Indeed, by Plancherel and H\"older, 
    \begin{equation}
        \|\widehat{h_>}|_{\eps = 1}\|_{\ell^2} \leq \|e^{F_{\mu, 1}(k)}\|_{\ell^\infty}\|\widehat{f_>}\|_{\ell^2} < \frac{\nu_{\rm{max}}}{2 }.
    \end{equation}
    Hence, we have that $\|f_>\|_{\mu, 1} < \frac{\nu_{\rm{max}}}{2}$.
    We then use the continuity argument as in \cite{EHL11}: the map $\eps \mapsto \|f_>\|_{\mu, \eps}$ is continuous on the interval $(0,1]$ we have \eqref{eq:g-b} holds for every $\eps \in (0,1]$, so $\|f_>\|_{\mu, \eps}$ has to stay in the same connected component of $G_\omega^{-1} (I )$ for $I=[0, g_0(\omega)]$ as $\|f_>\|_{\mu, 1}$, which is $[0,\frac{\nu_{\rm{max}}}{2}]$. Using the monotone convergence theorem and Plancherel's theorem, we conclude that 
    \begin{equation}
        \|e^{F_{\mu, 0}(k)}\hat{f}(k)\|_{\ell^2}\leq \frac{\nu_{\rm{max}}}{2 },
    \end{equation}
    which proves the claim. The pointwise estimates follow in the same way as in \cite{EHL11}, by using $1d$ Sobolev embedding.
\end{proof}
We remark that the condition \eqref{eq:cond-L-exp} becomes weaker if $\omega$ grows. An easy computation shows that $\nu_{\rm{max}}$ is strictly increasing in $\omega$, then that $g_0(\omega)$ is strictly increasing, which implies that $\omega \mapsto \min(\tfrac{\omega}{2},g_0(\omega))$ is strictly increasing.

\section{Asymptotics as $L \to \infty$}
\label{sec:largeL}

In this section, we analyze the behaviour of $\omega_1(L)$ as $L \to \infty$. We will prove an asymptotic lower bound.

The optimal constant $\omega_1^*$ was introduced in \eqref{maxdefR}. In this section, we abuse notation and use $\cQ$ for the functional on the line, so that

$$\omega_1^*=\sup_{\|\phi\|_{L^2}=1} \cQ(\phi,\phi,\phi,\phi)=\max_{\|\phi\|_{L^2}=1} \int_0^1\int_\R \Big|T_r(\phi)\Big|^4 dxdr.$$
By the $L^4$ Strichartz estimate on $[0,1]\times \R$ it is finite, it is obviously positive and the maximum is attained, see \cite{HL-var,HL12b,HL09,EHL11}.

\begin{pro}\label{pro:unif-lower}
The periodic constant is asymptotically bounded below by the constant on the line. We have
$$\liminf_{L\to \infty} \omega_1(L)\geq \omega_1^*>0.$$
\end{pro}

\begin{proof}
Let $\psi$ be a maximizer for $ \omega_1^*$. By \cite{HL-var,HL12b,HL09,EHL11}, it exists, it is analytic and $\psi$ and $\hat{\psi}$ decay exponentially. Let
$$\psi_L(x):=\sum_{k \in \Z} \psi(x+kL)$$
be its periodization with period $L$. If $T_r$ denotes the free propagator, both on $\T_L$ and on $\R$, then we observe that
$T_r \psi_L=(T_r\psi)_L$, so the periodization commutes with the free evolution. In addition, for any $N\in \N$,
$$\sup_{r \in [0,1], |x|\leq L/2}|T_r \psi_L(x)-T_r\psi(x)|=\sup_{r \in [0,1], |x|\leq L/2}|\sum_{k\ne 0} (T_r\psi)(x+kL)|\lesssim_N L^{-N}.$$
Also, for any $p\geq 1$, we have $\sup_{r \in [0,1]}\|T_r\psi\|_{L^p(|x|\geq L/2)}\lesssim_N L^{-N}$. As a consequence, we obtain
$$\lim_{L\to \infty}\|T_r (\psi_L)\|_{L^4([0,1]\times \T_L)}=\|T_r \psi\|_{L^4([0,1]\times \R)}$$
and $\lim_{L \to \infty}\|\psi_L\|_{L^2(\T_L)}=\|\psi\|_{L^2(\R)}=1$.
Define $\phi_L:=\|\psi_L\|_{L^2(\T_L)}^{-1}\psi_L$. Then,
$$\omega_1(L)\geq \cQ(\phi_L,\phi_L,\phi_L,\phi_L)\to \omega_1^\ast \text{ as }L \to \infty,$$
and the proof is complete.
\end{proof}

\begin{remark}\label{rem:cons}
    \begin{enumerate}
    \item
    On $\T_L$, for the mass-$1$ constant solution $z_L=L^{-\frac12}$ we have
    $\cQ(z_L,z_L,z_L,z_L)=L^{-1}$.
    \item For sufficiently large $L$, the constant $z_L$ cannot be a maximizer. This proves Theorem \ref{thm:nontrivial}.
    \item 
We conjecture that we actually have convergence $\lim_{L \to \infty}\omega_1(L)=\omega_1^*$, but a sharp upper bound remains open. A (probably non-sharp) upper bound for $T=2\pi$ is provided in Subsection \ref{subsec:ub} below.
\end{enumerate}
\end{remark}

\subsection{An illustrative calculation for $T = 2 \pi$, $L = 2 \pi$}\label{subsec:ub}
Here, we follow the argument of Bourgain \cite{B93}, which is based on the construction of Zygmund \cite{zygmund1974fourier}. Letting
$u(x,t) = \frac{1}{\sqrt{2 \pi}} \sum_{k \in \Z} e^{i k x - i k^2 t} \hat{f} (k)$, by Parseval we have
\begin{align*}
\| u \|_{L^4 ( [0,2 \pi]^2)}^4 & = \| u (x,t) \overline{u} (x,t) \|_{L^2 ([0,2 \pi]^2)}^2 
 = \sum_{\tau, n} | \widehat{ u \overline{u} } (\tau, n) |^2  \\
 &=  \sum_{\tau, n}  \Big|  \sum_{ {k_1, k_2}  \atop{ n = k_1 - k_2, \atop{ \tau = k_2^2 - k_1^2} }} \hat{f} (k_1) \overline{ \hat{f} (k_2) }  \Big|^2 \\ &
\leq 2 \left[ \left( \sum_k | \hat{f} (k) |^2 \right) \right]^2 = 2  \| f \|_{L^2}^4.
\end{align*}
In the last inequality we have used that $\tau=-n(k_1+k_2)$, so for $n\ne 0$ the two constraints determine the pair $(k_1,k_2)$ uniquely, which gives the contribution $$\sum_{k_1\ne k_2} | \hat{f} (k_1) |^2|\hat{f} (k_2) |^2\leq \big(\sum_{k} | \hat{f} (k) |^2\big)^2,$$ and $n=0$ yields $k_1=k_2$, which contributes another $\big(\sum_{k} | \hat{f} (k) |^2\big)^2$.
Hence, the optimal Strichartz constant is at most $2^{\frac14}$.  Note that the constant function $f = 1$ shows that the optimal Strichartz constant actually lies between $1$ and $2^{\frac14} $.

\section{Numerics}
\label{sec:num}

\subsection{Spectral Renormalization and Discretization}

The method we derive here is based upon the spectral renormalization algorithm first proposed in the paper \cite{AbMu-specren}.  There, the authors work with the elliptic equation of the form
\begin{equation}
-\Delta u + \mu u + V(x) u - F(|u|^2) u = 0
\end{equation}
by rewriting it in the form
\begin{equation}
 \hat u (k) = - \frac{ \mathcal{F} (V u) }{|k|^2 + \mu} + \frac{ \mathcal{F} ( N( |u|^2 u) ) }{ |k|^2 + \mu}.
 \end{equation}
 However, due to scaling symmetries implicit in any constrained minimization algorithm, this particular method is unstable.  To fix that, define $\mu \hat w = \hat u$ and set up the iterative algorithm 
 \begin{equation}
 \hat{w}_{m+1} = - \frac{ \mathcal{F} (V w_m) }{|k|^2 + \mu} + \frac{ \mathcal{F} ( N( |\mu_m^2 w_m|^2 w_m) ) }{ |k|^2 + \mu}
 \end{equation}
 where $\mu_m$ is selected such that
\begin{align}
G( \mu_m) & = \int | \hat w_m |^2 dk - \int \bar{\hat{w}} \left[  \frac{ \mathcal{F} (V w_m ) }{|k|^2 + \mu} + \frac{ \mathcal{F} ( N( | \mu_m^2 w|^2 w) ) }{ |k|^2 + \mu} \right] dk \\
&= 0.
\end{align}
Such an iterative algorithm converges quite quickly when the potential has reasonable decay in frequency space and when the elliptic part is sufficiently coercive.

To handle the case of the dispersion managed ($0$ average dispersion), we follow a similar track, though the only source of ellipticity is now that of the $Q$ functional.  To that end, we set up the algorithm as follows.  Let us attempt to solve the equation \eqref{E:0av-weak-dmnls}
\begin{equation}
\omega u = Q (u,u,u),
\end{equation}
where 
\begin{equation}
Q(u,u,u) = \int_0^1 T_r^{-1} ( |T_r u|^2 \overline{T_r u} ) dr.
\end{equation}
Looking in Fourier space, we then observe
\begin{equation}
\omega \hat u (k) = \int_0^1 e^{-i r |k|^2} \mathcal{F}\Big( \big| \mathcal{F}^{-1} \big(e^{-i r |k|^2} \hat u\big)\big|^2  \mathcal{F}^{-1} \Big( e^{-i r |k|^2} \hat u \big) \Big)dr .
\end{equation}
Again, in order to overcome scaling symmetries in the $L^2$ norm, we redefine $\hat u = \gamma \hat w$.  As a result, we can define the algorithm
\begin{equation}
\hat{w}_{m+1} = \mu^{-1} \gamma_m^2
\int_0^1 e^{-i r |k|^2} \mathcal{F}\Big( \big| \mathcal{F}^{-1} \big(e^{-i r |k|^2} \hat w_m\big)\big|^2  \mathcal{F}^{-1} \Big( e^{-i r |k|^2} \hat w_m \big) \Big)dr,
\end{equation}
where
\begin{equation}
\gamma_m^2 = \frac{ \mu    \int | \hat w_m |^2 dk }{ \int_0^1  \int | \mathcal{F}^{-1}  \big(e^{i r |k|^2} \hat w_m \big) |^4  dx dr }.
\end{equation}

Such an iterative method converges, where there are several choices one must make in the discretization.  For instance, one must choose a means of approximating the integral over $r$ in the functional $Q$.  To do so, one can either use implicit quadrature methods from {\it Matlab} and set the tolerance level to be say $tol = 1e-12$ or one can choose a discretization in the $r$ variables with $r_{\rm max}$ grid points over the interval $0 \leq r \leq 1$ and compute the integral through direct approximation.  Typically we implement the latter and set $r_{\rm max} = 1e3$.
 
\subsection{Numerical Simulations}

We have run the algorithm above with initial conditions that have symmetric, anti-symmetric and asymmetric initial inputs as periodic extensions of the function
\begin{equation}
u_0 = \alpha x e^{-(x- \gamma )^2/4}/ \sqrt{4\pi} + \beta e^{-x^2 / 4}/\sqrt{4\pi},
\end{equation}
for $\alpha = 0$, $\beta =1$ representing symmetric data and $\alpha = 1, \beta = 0$ representing anti-symmetric.  If both $\alpha,\beta \neq 0$, we are considering general periodic solutions.  
The resulting output is contained in Figures \ref{fig:dmsol1}-\ref{fig:dmsol1_comp}.  We observe that up to symmetries, generically a symmetric configuration is selected by the numerical method, unless we restrict to purely anti-symmetric initial data.  With purely anti-symmetric data, we can construct an anti-symmetric stationary solution, but it is of lower energy over all than the symmetric state selected by generic simulation.

\begin{figure}[!htbp]
\includegraphics[scale=0.3]{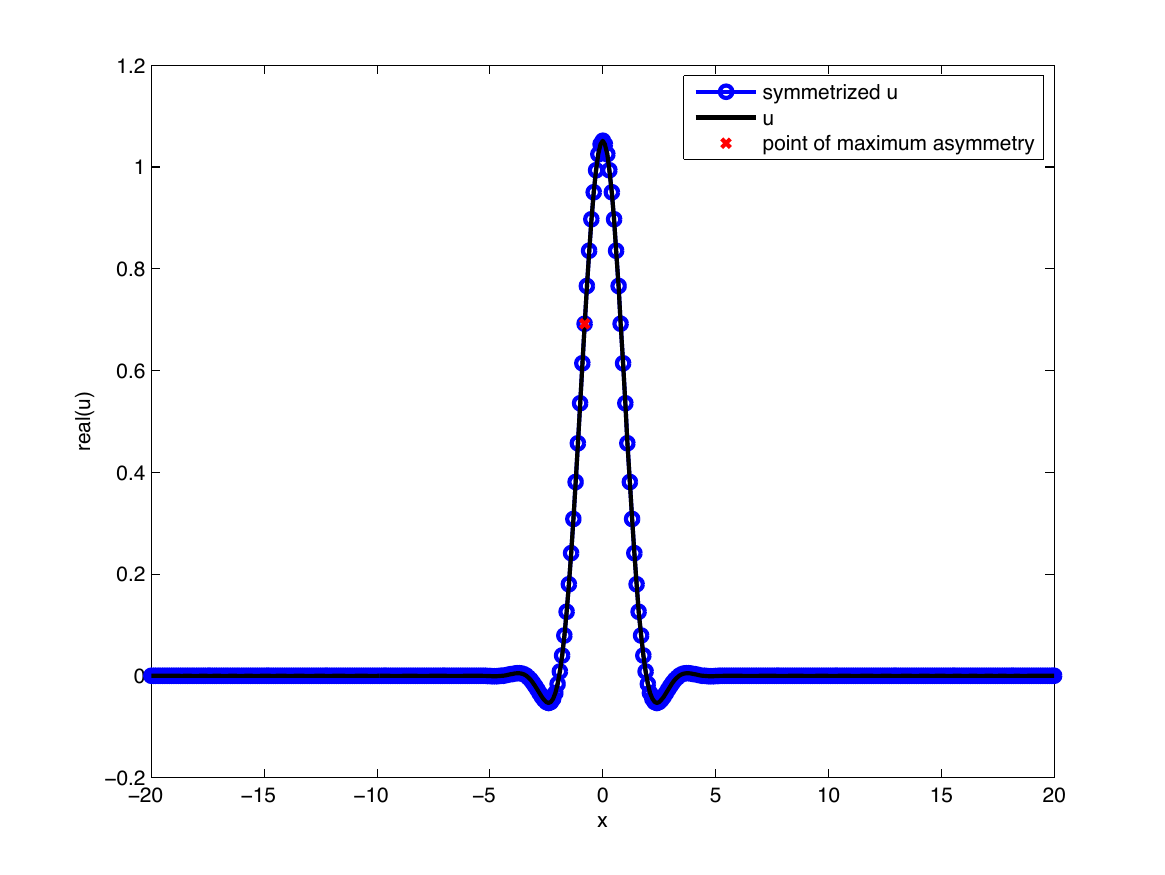}
\caption{Converged numerical soliton for $\mu = 1.0$, $L = 40.0$, $N = 200$, symmetric initial iterate ($\alpha = 0.0$, $\beta = 1.0$).  Resulting dispersion managed soliton has converged to a bound state up to order at least $1e-9$, has numerical $L^2$ norm squared $1.6806$ and is even to machine precision.}
\label{fig:dmsol1}
\end{figure}

\begin{figure}[!htbp]
\includegraphics[scale=0.3]{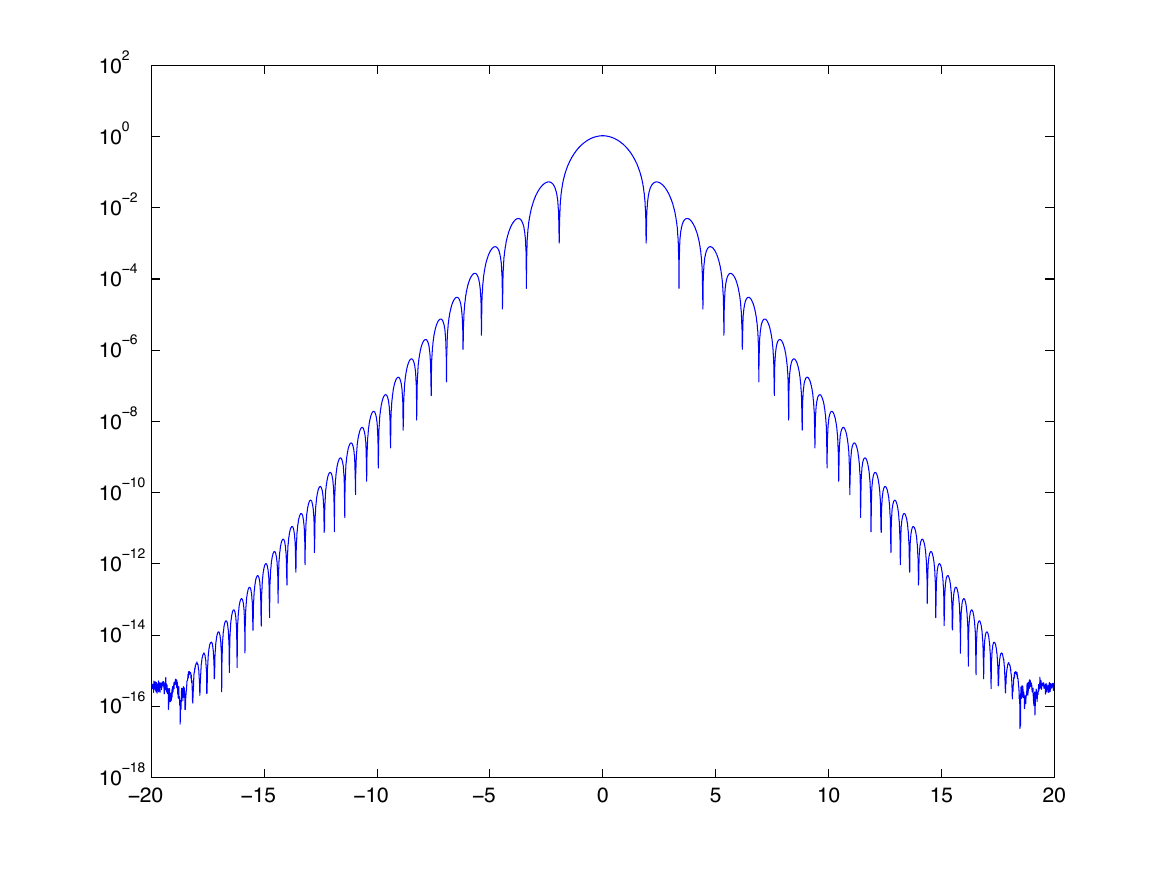}
\caption{Converged numerical soliton in a semilog plot for $\mu = 1.0$, $L = 40.0$, $N = 2000$, symmetric initial iterate ($\alpha = 0.0$, $\beta = 1.0$).  Resulting Dispersion Managed Soliton has converged to a bound state up to order at least $1e-9$, has numerical $L^2$ norm squared $1.6806$ and is even up to order $1e-15$.}
\label{fig:dmsol1_semilog}
\end{figure}

\begin{figure}[!htbp]
\includegraphics[scale=0.3]{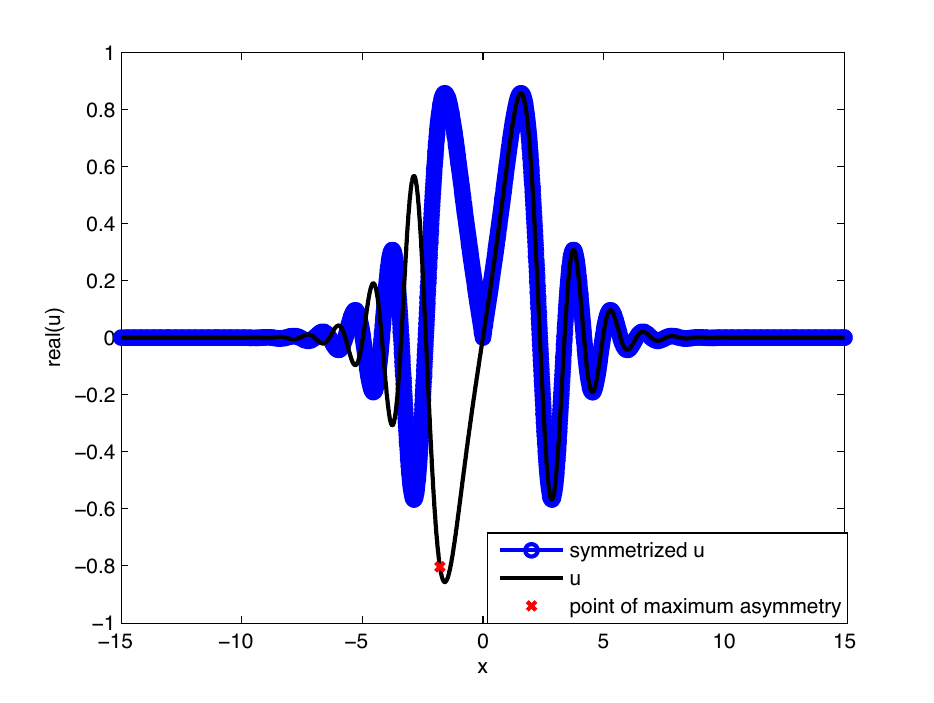}
\caption{Converged numerical soliton for $\mu = 1.0$, $L = 30.0$, $N = 1500$, antisymmetric initial iterate ($\alpha = 1.0$, $\gamma = 0.0$, $\beta = 0.0$).  Resulting Dispersion Managed Soliton has converged to bound state up to order at least $3e-4$, has numerical $L^2$ norm squared $3.135$ and is odd.}
\label{fig:dmsol2}
\end{figure}

\begin{figure}[!htbp]
\includegraphics[scale=0.3]{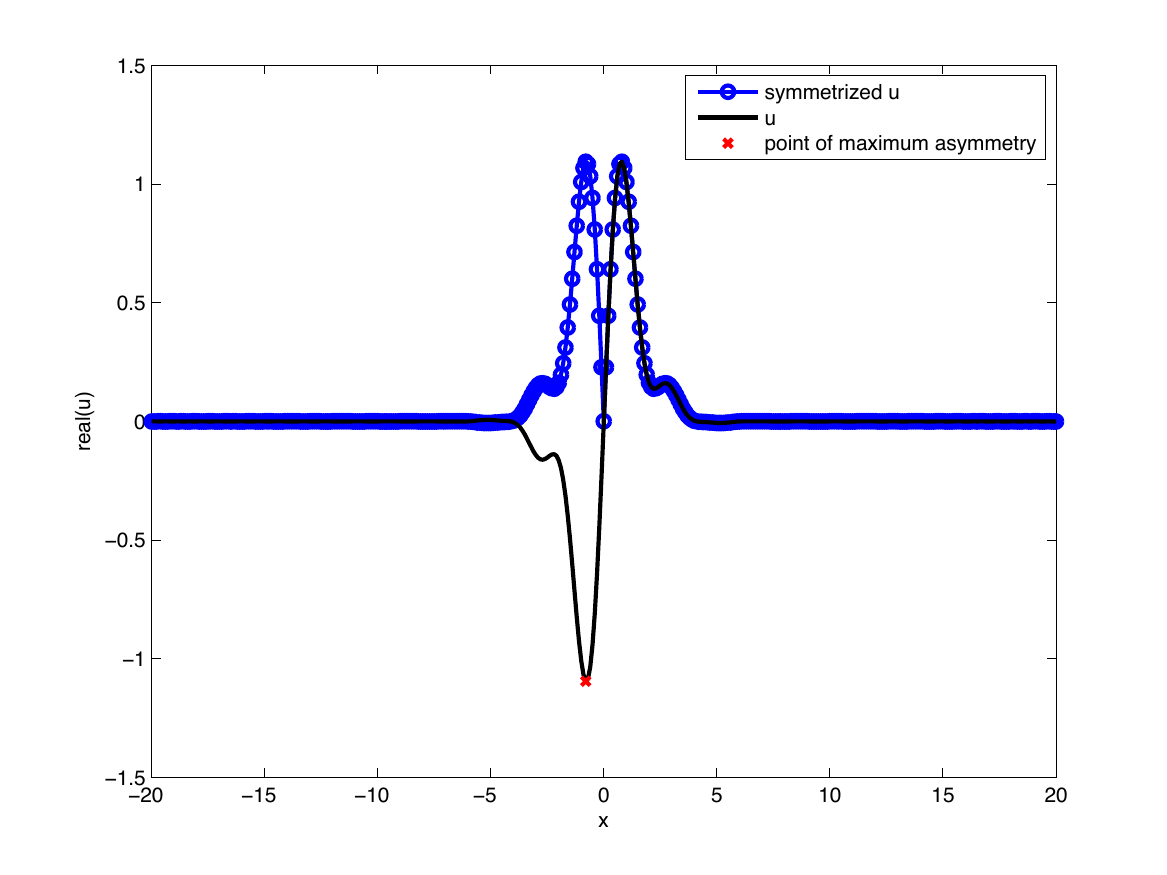}
\caption{Converged numerical soliton for $\mu = 1.0$, $L = 50.0$, $N = 2500$, antisymmetric initial iterate ($\alpha = 1.0$, $\gamma = 0.0$, $\beta = 0.0$).  Resulting Dispersion Managed Soliton has converged to bound state up to order at least $1e-6$, has numerical $L^2$ norm squared $2.3713$ and is odd.}
\label{fig:dmsol2alt}
\end{figure}

\begin{figure}[!htbp]
\includegraphics[scale=0.3]{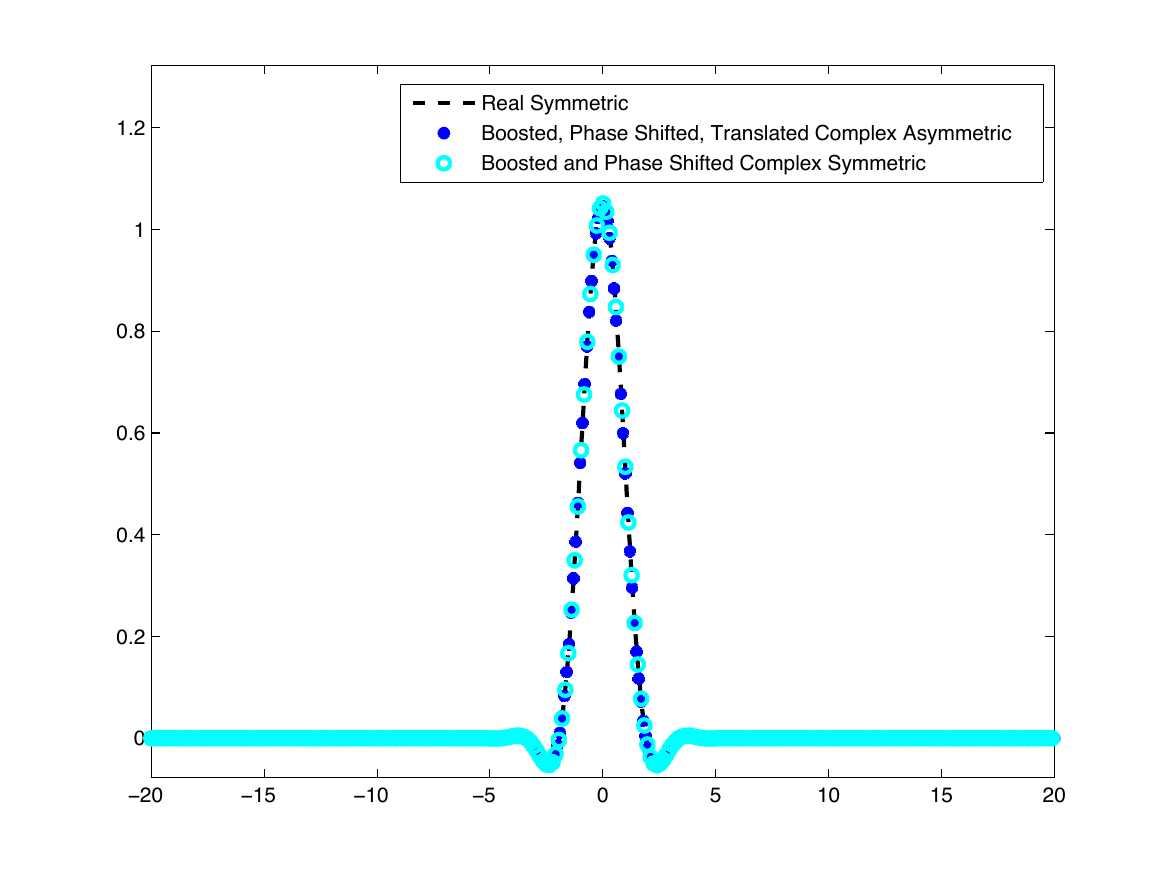}
\caption{Comparison of maximizers with real symmetric initial state, complex symmetric initial state and complex asymmetric initial state after removing spatial translation, frequency boosts and phase shifts.}
\label{fig:dmsol1_comp}
\end{figure}
\clearpage

\bibliographystyle{alpha}
\bibliography{HM-bib}

\end{document}